\documentclass [11pt] {article}
\usepackage{amsfonts}
\usepackage{amssymb}

\newcommand{\qed} {\hspace {0.1in} \rule {1.5mm} {3.5mm}}

\newtheorem{lemma}{Lemma}[section]

\newtheorem{theorem}{Theorem}

\newtheorem{proposition}{Proposition}[section]
\newtheorem{definition}{Definition}[section]
\def\o{\omega}

\def\e{\epsilon}

\def\limo{\lim_{\omega}}

\def\<{\langle}
\def\>{\rangle}

\def\proof{\smallskip\noindent{\it Proof.} }

\def\bH{{\bf H}}
\def\bI{{\bf I}}
\def\bJ{{\bf J}}
\def\bR{{\mathbb R}}
\def\bN{{\mathbb N}}\def\bQ{{\mathbb Q}}

\def\bS{{\bf S}}

\def\cA{\mbox{$\cal A$}}
\def\cB{\mbox{$\cal B$}}
\def\cC{\mbox{$\cal C$}}

\def\cH{\mbox{$\cal H$}}

\def\cA{\mbox{$\cal A$}}
\def\cB{\mbox{$\cal B$}}
\def\cC{\mbox{$\cal C$}}
\def\cP{\mbox{$\cal P$}}

\def\cM{\mbox{$\cal M$}}
\def\cN{\mbox{$\cal N$}}
\def\cS{\mbox{$\cal S$}}

\def\to{\rightarrow}

\def\xo{{\bf X}}
\def\bo{\cB_\omega}

\def\muo{\mu}

\def\xok{\xo^{[k]}}
\def\ac{A^c}
\def\wch{\widetilde{\cH}}
\def\wp{\widetilde{P}}

\begin{document}

\title{Limits of Hypergraphs, Removal and Regularity Lemmas. A 
Non-standard Approach}
\author{{\sc G\'abor Elek} and {\sc Bal\'azs Szegedy}}

\maketitle

\abstract{We study the integral and measure theory of the ultraproduct of
finite sets. As a main application we construct limit objects for
hypergraph sequences. We give a new proof for the Hypergraph Removal
Lemma and the Hypergraph Regularity Lemma.}

\tableofcontents

\section{Introduction}
The so-called Hypergraph Regularity Lemma (R\"odl-Skokan \cite{RSko},
Gowers \cite{Gow}, later generalized by Tao \cite{Tao}) is one of the 
most exciting result
in modern combinatorics. It exists in many different forms, strength
and generality. The main message in all of them is that every
$k$-uniform hypergraph can be approximated by a structure which
consists of boundedly many random-looking (quasi-random) parts for any
given error $\epsilon$. Another common feature of these theorems is
that they all come with a corresponding counting lemma \cite{NRS}
which describes how to estimate the frequency of a given small
hypergraph from the quasi-random approximation of a large hypergraph.
One of the most important applications of this method is that it
implies the Hypergraph Removal Lemma (first proved by
Nagle, R\"odl and Schacht \cite{NRS})  and by an
observation of Solymosi \cite{S} it also 
implies Szemer\'edi's celebrated theorem on
arithmetic progressions in dense subsets of the integers even in a
multidimensional setting.

 In this paper we present an analytic approach to the subject.
First, for any given sequence of hypergraphs we associate the
so-called ultralimit hypergraph,
which is a measurable hypergraph in a large (non-separable)
probability measure space. 
The ultralimit method enables us to convert theorems of finite
combinatorics to measure theoretic statements on our ultralimit
space. In the second step, using separable approximations
we translates these measure-theoretic
theorems to well-known results on the more familiar Lebesgue
spaces. This way in two steps we prove the Hypergraph Removal Lemma
from the Lebesgue Density Theorem and the Hypergraph Regularity
Lemma from the Rectangular Approximation Lemma of Lebesgue Spaces.

\noindent
 We also construct a Hypergraph Limit Object to
convergent hypergraph sequences directly from the
ultralimit hypergraph. This construction is the generalization of the
limit graph method \cite{Borgs},\cite{LSZ}
where limits of sequences of
dense graphs are studied. According to a definition by Borgs
et. al. \cite{Borgs} a
graph sequence is called convergent if the density of any fixed graph
in the terms of the sequence is convergent. In a paper by Lovasz and
Szegedy \cite{LSZ} it is shown that a convergent graph sequence has a natural
limit object which is a two variable function $w:[0,1]^2\to[0,1]$ with
$w(x,y)=w(y,x)$. Informally speaking, $w$ is an infinite analogue of
the adjacency matrix. Our main theorem is a generalization of this
theorem to $k$-uniform hypergraphs. We also
show that limits of $k$-uniform hypergraphs can be represented by $2^k-2$
variable measurable functions $w:[0,1]^{2^k-2}\to[0,1]$ such that the
coordinates are indexed by the proper non empty subsets of $\{1,2,\dots,k\}$
and $w$ is invariant under the induced action of $S_k$ on the coordinates.

\vskip0.2in
\noindent
{\bf Acknowledgement:} We are very indebted to 
Terence Tao and L\'aszl\'o Lov\'asz for
helpful discussions.

\section{Analysis on the ultraproduct of finite measure spaces}

\subsection{Ultraproducts of finite sets}\label{sec1}
First we recall the ultraproduct construction of finite probability
measure spaces (see \cite{Loeb}).
Let $\{X_i\}^\infty_{i=1}$ be finite sets.
We always suppose that $|X_1|<|X_2|<|X_3|<\dots$ Let $\omega$ be a nonprincipal
ultrafilter and $\lim_{\o}:l^\infty(\bN)\to\bR$ be the corresponding 
ultralimit. Recall that $\limo$ is a bounded linear functional
such that for any $\epsilon>0$ and $\{a_n\}_{n=1}^\infty\in l^\infty(\bN)$
$$\{ i\in \bN\,\mid\, a_i\in [\limo a_n-\e, \limo a_n +e]\}\in\omega\,.$$
The ultraproduct of the sets $X_i$ is defined as follows.

\noindent
Let $\widetilde{X}=\prod^\infty_{i=1}X_i$. We say that
$\widetilde{p}=\{p_i\}^\infty_{i=1}, \widetilde{q}=\{q_i\}^\infty_{i=1}\in
\widetilde{X}$ are equivalent, 
$\widetilde{p}\sim\widetilde{q}$, if
$$\{i\in \bN\mid p_i=q_i\}\in \omega\,.$$
Define $\xo:=\widetilde{X}/\sim$.
Now let $\cP(X_i)$ denote the Boolean-algebra of subsets of $X_i$, with the
normalized measure $\mu_i(A)=\frac{|A|}{|X_i|}\,.$
Then let $\widetilde{\cP}=\prod^\infty_{i=1}\cP(X_i)$ and
$\cP=\widetilde{P}/I$, where $I$ is the ideal of elements
$\{A_i\}^\infty_{i=1}$
such that 
$\{i\in \bN\mid A_i=\emptyset\}\in \omega\,.$
Notice that the elements of $\cP$ can be identified with certain subsets
of $\xo$: If 
$$\overline{p}=[\{p_i\}^\infty_{i=1}]\in \xo\,\,\mbox{and}\,\, \overline{A}=
[\{A_i\}^\infty_{i=1}]\in \cP$$
then $\overline{p}\in \overline{A}$ if 
$\{i\in \bN\mid p_i\in A_i\}\in \omega\,.$
Clearly, if $\overline{A}=
[\{A_i\}^\infty_{i=1}]$, $\overline{B}=
[\{B_i\}^\infty_{i=1}]$ then
\begin{itemize}
\item 
$\overline{A}^c=[\{A^c_i\}^\infty_{i=1}]\,,$
\item
$\overline{A}\cup \overline{B}=[\{A_i\cup B_i\}^\infty_{i=1}]\,,$
\item
$\overline{A}\cap \overline{B}=[\{A_i\cap B_i\}^\infty_{i=1}]\,.$
\end{itemize}
That is $\cP$ is a Boolean algebra on $\xo$. 
Now let $\muo(\overline{A})=\lim_{\o} \mu_i(A_i)$. Then $\muo:\cP\to\bR$ is
a finitely additive probability measure.
\begin{definition}
$N\subseteq \xo$ is a {\bf nullset} if for any $\e>0$ there exists
a set $\overline{A_\e}\in\cP$ such that $N\subseteq \overline{A_\e}$
and $\muo(\overline{A_\e})\leq \e$. The set of nullsets is denoted
by $\cN$.
\end{definition}
\begin{proposition}
$\cN$ satisfies the following properties:
\begin{itemize}
\item if $N\in \cN$ and $M\subseteq N$, then $M\in \cN$.
\item If $\{N_k\}^\infty_{k=1}$ are elements of $\cN$ then
$\cup^\infty_{k=1} N_k\in \cN$ as well.
\end{itemize} \end{proposition}
\begin{proof}
We need the following lemma.
\begin{lemma}\label{l6}
If $\{\overline{A_k}\}^\infty_{k=1}$ are elements of $\cP$
and $\lim_{l\to\infty} \muo(\cup^l_{k=1}\overline{A_k})=t$  then
there exists an element $\overline{B}\in\cP$ such that
$\muo(\overline{B})=t$ and $\overline{A_k}\subseteq \overline{B}$
for all $k\in \bN$.
\end{lemma}
\begin{proof}
Let $\overline{B_l}=\cup^l_{k=1}\overline{A_k}$, $\muo(\overline{B_l})=t_l$,
$\lim_{l\to\infty} t_l=t\,.$ Let
$$T_l=\left\{i\in\bN\,\mid\,
|\mu_i(\cup^l_{k=1} A^i_k)-t_l|\leq \frac{1}{2^l}\,\right\}\,,$$
where $\overline{A_k}=[\{A^i_k\}^\infty_{i=1}]\,.$
Observe that $T_l\in \omega$. If $i\in \cap^m_{l=1}T_l$ but
$i\notin T_{m+1}$, then let $C_i=\cup^m_{k=1} A^i_k\,.$
If $i\in T_l$ for all $l\in \bN$, then clearly $\mu_i(\cup^\infty_{k=1}
A^i_k)=t$ and
 we set $C_i:=\cup^\infty_{k=1} A^i_k\,.$
Let $\overline{B}:=[\{C_i\}^\infty_{i=1}]\,.$ Then
$\muo(\overline{B})=t$ and for any $k\in\bN$:
$\overline{A_k}\subseteq \overline{B}$. \qed \end{proof} \vskip 0.2in

\noindent
Now suppose that for any $j\geq 1$, $\overline{A_j}\in\cN$. Let
$\overline{B}^\e_j\in\cP$ such that $\overline{A_j}\subseteq
\overline{B}^\e_j$ and $\muo(\overline{B}^\e_j)<\e\frac{1}{2^j}$.
Then by the previous lemma, there exists $\overline{B}^\e\in\cP$ such that 
for any $j\geq 1$ 
$\overline{B}^\e_j\subseteq \overline{B}^\e$ and $\muo(\overline{B}^\e)\leq\e$.
Since $\cup^\infty_{j=1} \overline{A_j}\subseteq \overline{B}^\e$, our
proposition follows. \qed \end{proof} \vskip 0.2in
\begin{definition}
We call $B\subseteq \xo$ a {\bf measureable set} if there exists 
$\widetilde{B}\in \cP$
such that $B\triangle \widetilde{B}\in \cN$.
\end{definition}
\begin{theorem}
The measurable sets form a $\sigma$-algebra $\bo$ and $\muo(B)=
\muo(\widetilde{B})$
defines a probability measure on $\bo$.
\end{theorem}
\begin{proof}
We call two measurable sets $B$ and $B'$ equivalent, $B\cong B'$ if 
$B\triangle B'\in \cN$.
Clearly, if $A\cong A'$, $B\cong B'$ then $A^c\cong (A')^c$,
$A\cup B\cong A'\cup B'$, $A\cap B\cong A'\cap B'$. Also if
$A,B\in\cP$ and $A\cong B$, then $\muo(A)=\muo(B)$. That is
the measurable sets form a Boolean algebra with a finitely additive measure.
Hence it is enough to prove that if $\overline{A_k}\in\cP$ are disjoint sets,
then there exists $\overline{A}\in\cP$ such that
$\cup^\infty_{k=1}\overline{A_k}\cong \overline{A}$ and
$\muo(\overline{A})=\sum^\infty_{k=1}\muo(\overline{A_k})\,.$
Note that by Lemma \ref{l6} there exists $\overline{A}\in\cP$ such that
$\muo(\overline{A})=\sum^\infty_{k=1}\muo(\overline{A_k})$ and
$\overline{A_k}\subseteq \overline{A}$ for all $k\geq 1$.
Then for any $j\geq 1$,
$$\overline{A}\backslash \cup^\infty_{k=1}\overline{A_i}\subseteq
 \overline{A}\backslash \cup^j_{k=1}\overline{A_k}\in \cP\,.$$
Since $\lim_{j\to\infty}\muo(\overline{A}\backslash
\cup^j_{k=1}\overline{A_k})=0,  
\overline{A}\backslash \cup^\infty_{k=1}\overline{A_k}\in \cN$ thus
$\cup^\infty_{k=1}\overline{A_k}\cong
\overline{A}$.\qed \end{proof} \vskip 0.2in

\noindent
Hence we constructed an atomless probability measure space $(\xo,\bo,\muo)$.
Note that this space is non-separable, that is it is not measurably 
equivalent to the interval with the Lebesgue measure.

\subsection{Measureable functions and their integrals}
Let $\{X_i\}_{i=1}^\infty$ be finite sets as in the previous section and 
$f_i:X_i\to[-d,d]$ be real functions, where $d>0$. Then one can define
a function $f:\xo\to[-d,d]$ whose value at $\overline{p}=
[\{p_i\}^\infty_{i=1}]$
is the ultralimit of $\{f_i(p_i)\}^\infty_{i=1}$. We say that
$f$ is the ultralimit of the functions $\{f_i\}^\infty_{i=1}$. From now on
we call such bounded functions {\bf ultralimit functions}.
\begin{lemma}\label{ultralimit}
The ultralimit functions are  measurable on $\xo$ and 
$$\int_{\xo} f d\,\muo=\limo \frac{\sum_{p\in X_i} f_i(p)}{|X_i|}\,.$$
\end{lemma}
\begin{proof} Let $-d\leq a\leq b\leq d$ be real numbers. It is enough to
prove that
$f_{[a,b]}=\{\overline{p}\in \xo\mid\, 
a\leq f(\overline{p})\leq b\}$ is measurable.
Let $f_{[a,b]}^i=\{p\in X_i\mid a\leq f_i(p)\leq b\}\,.$
Note that $[\{f_{[a,b]}^i\}^\infty_{i=1}]$ is
not necessarily equal to $f_{[a,b]}$. Nevertheless if
$$P_n:= [\{f^i_{[a-\frac{1}{n},b+\frac{1}{n}]}\}^\infty_{i=1}]\,,$$
then $P_n\in\cP$ and $f_{[a,b]}=\cap^\infty_{n=1} P_n$. This shows that
$f_{[a,b]}$ is a measurable set. Hence the function $f$ is measurable.
Now we prove the integral formula.
Let us consider the function $g_i$ on $X_i$ which takes the value
$\frac{j}{2^k}$ if $f_i$ takes a value not greater than
$\frac{j}{2^k}$ but less than $\frac{j+1}{2^k}$ for 
$-N_k\leq j \leq N_k$, where $N_k=[d 2^k]+1$.
Clearly
$ |\limo g_i - f|\leq \frac{1}{2^k}$ on $\xo$. Observe that
$g=\limo g_i$ is a measurable step-function on $\xo$ taking the
value $\frac{j}{2^k}$ on $C_j= [\{f^i_{[\frac{j}{2^k}, 
\frac{j+1}{2^k})}\}^\infty
_{i=1}]$. Hence,
$$\int_X g\,d\muo=\limo\left(\sum^{N_k}_{j=-N_k} 
\frac{|f^i_{[\frac{j}{2^k},\frac{j+1}{2^k})}|}{|X_i|}\frac{j} {2^k}\right)\,.$$
Also, $|g-f|\leq\frac{1}{2^k}$ on $\xo$ uniformly, that is
$|\int_{\xo} f\,d\muo - \int_{\xo} g\,d\muo|\leq \frac{1}{2^k}\,.$
Notice that for any $i\geq 1$
$$\left|\sum^{N_k}_{j=-N_k} 
\frac{|f^i_{[\frac{j}{2^k},\frac{j+1}{2^k})}|}{|X_i|}\frac{j} {2^k}-
\frac{\sum_{p\in X_i} f_i(p)}{|X_i|}\right|\leq \frac{1}{2^k}\,.$$
Therefore for each $k\geq 1$,
$$\left|\int_\xo f\,d\muo-\limo \frac{\sum_{p\in X_i} f_i(p)}{|X_i|}\right|\leq
\frac{1}{2^{k-1}}\,.$$
Thus our lemma follows.
\qed
\end{proof} \vskip 0.2in

\begin{theorem}\label{tetel2}
For every measurable function $f:\xo\to[-d,d]$, there exists a sequence
of functions $f_i:X_i\to [-d,d]$ such that
the ultralimit of the sequence $\{f_i\}_{i=1}^\infty$ is
almost everywhere equals to $f$. That is any element of
$L^\infty(\xo,\bo,\muo)$ can be represented by an ultralimit function.
\end{theorem}
\begin{proof}
Recall a standard result of measure theory. If $f$ is a bounded measurable
function on $\xo$, then there exists a sequence of bounded
stepfunctions $\{h_k\}^\infty_{k=1}$ such that
\begin{itemize}
\item
$f=\sum^\infty_{k=1} h_k$
\item  $|h_k|\leq \frac{1}{2^{k-1}}$, if $k>1$.
\item $h_k=\sum^{n_k}_{n=1} c^k_n \chi_{A^k_n}$, where 
$\cup^{n_k}_{n=1}A^k_n=\xo$ is a
measurable partition, $c^k_n\in\bR$ if $1\leq n \leq n_k$.
\end{itemize}
Now let $B^k_n\in\cP$ such that $\muo(A^k_n\triangle B^k_n)=0$.
We can suppose that $\cup^{n_k}_{n=1} B^k_n$ is a partition of $\xo$.
Let $h'_k=\sum^{n_k}_{n=1} c^k_n \chi_{B^k_n}$ and
$f'=\sum^\infty_{k=1} h'_k$.
Then clearly $f'=f$ almost everywhere. We show that $f'$ is an ultralimit
function.

\noindent
Let $B^k_n=[\{B^k_{n,i}\}^\infty_{i=1}]$. 
We set $T_k\subset \bN$ as the set of integers $i$ for which 
$\cup_{n=1}^{n_k} B^k_{n,i}$ is a partition of $X_i$. Then obviously,
$T_k\in\omega$.
Now we use our diagonalizing trick again. If $i\notin T_1$ let $s_i\equiv 0$.
If $i\in T_1, i\in T_2,\dots,i\in T_k, i\notin T_{k+1}$ then
define $s_i:=\sum^k_{j=1}(\sum^{n_j}_{n=1} c^j_n \chi_{B^j_{n,i}})\,.$
If $i\in T_k$ for each $k\geq 1$ then set
$s_i:=\sum^i_{j=1}(\sum^{n_j}_{n=1} c^i_n \chi_{B^j_{n,j}})\,.$
Now let $\overline{p}\in B^1_{j_1}\cap B^2_{j_2}\cap\dots\cap B^k_{j_k}$.
Then
$$|(\limo s_i)(\overline{p})-f'(\overline{p})|\leq\frac{1}{2^{k-1}}\,.$$
Since this inequality holds for each $k\geq 1$, $f'\equiv \limo s_i$. \qed
\end{proof} \vskip 0.2in
\subsection{Fubini's Theorem and the Integration Rule}
We fix a natural number $k$ and we denote by $[k]$ the set $\{1,2,\dots,k\}$.
 Let $X_{i,1},X_{i,2},\dots,X_{i,k}$ be $k$ copies of the finite set $X_i$ and
 for a subset
 $A\subseteq\{1,2,\dots,k\}$ let $X_{i,A}$ denote the direct
 product $\bigoplus_{j\in A}X_{i,j}$. 
Let $\xo^A$ denote the ultra product of the sets $X_{i,A}$, with a Boolean
algebra $\cP_A$.
 There is a natural
 map $p_A:\xok\to \xo^A$ (the projection). Let $\cB_A$ be the
$\sigma$-algebra of measurable subsets in $\xo^A$ as defined in the previous
sections. Define $\sigma(A)$ as $p_A^{-1}(\cB_A)$, the $\sigma$-algebra
of measurable sets depending only on the $A$-coordinates together with
the probability measure $\mu_A$. 
For a nonempty subset $A\subseteq [k]$ let $A^*$ denote the 
set system $\{B|B\subseteq
A~,~|B|=|A|-1\}$ and let $\sigma(A)^*$ denote the $\sigma$-algebra $\langle
\sigma(B)|B\in A^*\rangle$. An interesting fact is (as it will turn out in
subsection \ref{randpar}) that $\sigma(A)^*$ is strictly smaller 
than $\sigma(A)$. 

\begin{lemma}
\label{l14}
\label{vetites} Let $A,B\subseteq [k]$ and 
let $f:\xok\to\mathbb{R}$ be a bounded $\sigma(B)$-measurable ultralimit
 function.
 Then for all $y\in \xo^{\ac}$ the function 
$f_y$ is $\sigma(A\cap B)$-measurable, where $\ac$ denotes
the complement of $A$ in $[k]$ and $f_y(x)=f(x,y)$.
\end{lemma}
\begin{proof}
Let $f:\xok\to\bR$ be a $\sigma(B)$-measurable ultralimit
function. It is easy to
see that the finite approximation functions $f_i:X_{i,1}\times
X_{i,2}\times\dots X_{i,k}$ constructed in Theorem \ref{tetel2} depend only
on the $B$-coordinates, since $\sigma(B)$-measurable
functions can be approximated by $\sigma(B)$-measurable stepfunctions. 
Let $y\in\xo^{\ac}$,
$y=[\{y_i\}^\infty_{i=1}]$. Then $f_y$ is the ultralimit of 
the functions $f^{y_i}_i$. Clearly $f^{y_i}_i$ depends only on the $A\cap
B$-coordinates, thus the ultralimit $f_y$ is $\sigma(A\cap B)$-measurable.
\qed \end{proof} \vskip 0.2in
\begin{theorem}[Fubini's Theorem] Let $A\subseteq [k]$ and 
let $f:X_k\to\mathbb{R}$ be a $\sigma([k])$-measurable ultralimit function.
 Then 
$$\int_{\xok}f(p) d\mu_{[k]}(p)=\int_{\xo^{\ac}}
\left(\int_{\xo^A}f_y(x) d\mu_A(x)
\right) d\mu_{\ac} (y)$$
\end{theorem}
 
\begin{proof} 
Let $f$ be the ultralimit of $\{f_i:X_{i,[k]}\to\mathbb{R}\}^\infty_{i=1}$. 
Define the functions $\overline{f_i}:X_{i,\ac}\to [-d,d]$ by
$$\overline{f_i}(y)=|X_{i,A}|^{-1}\sum_{x\in X_{i,A}}f_i(x,y).$$
By Lemma \ref{ultralimit} 
$$\limo \overline{f_i}(y)=\int_{\xo^A} f(x,y) \,d\mu_A(x)\,.$$ 
Applying Lemma \ref{ultralimit}
again for the functions $\overline{f_i}$, we obtain that
$$\limo |X_{i,\ac}|^{-1}\sum_{y\in X_{i,\ac}}
\overline{f_i}(y)
=\int_{\xo^{\ac}}\left(\int_{\xo^A}f(x,y) d\mu_A(x)\right) d\mu_{\ac}(y)\,.$$
This completes the proof, since
$$|X_{i,\ac}|^{-1}\sum_{y\in X_{i,\ac}}
\overline{f_i}(y)=\frac{\sum_{p\in X_i} f_i(p)}{|X_i|}\,.$$
\qed \end{proof} \vskip 0.2in 
Recall that if $\cB\subset\cA$ are $\sigma$-algebras on $X$ with  a measure
$\mu$
 and
$g$ is an $\cA$-measurable function on $X$, then $E(g\mid \cB)$ is 
the $\cB$-measurable function (unique up to a zero measure perturbation)
 with the property
that
$$\int_Y E(g\mid \cB)\,d\mu=\int_Y g \,d\mu\,,$$
for any $Y\in\cB$ (see Appendix).
\begin{theorem}[Integration Rule]
Let $g_i:\xok\to\mathbb{R}$ be bounded 
$\sigma(A_i)$-measurable functions for $i=1,2,\dots,m$. Let $B$ denote the
sigma algebra generated by 
$\sigma(A_1\cap A_2),\sigma(A_1\cap A_3),\dots,\sigma(A_1\cap A_m)$. Then
$$\int_{\xok} g_1g_2\dots g_m\,d\mu_{[k]}=
\int_{\xok} E(g_1|B)g_2g_3\dots g_m\,d\mu_{[k]}\,.$$
\end{theorem}

\begin{proof} We may suppose that
all $g_i$ are ultralimit functions, since the conditional
expectation does not depend on zero measure perturbation.
Since $g_1$ does not depends on the $A_1^c$ coordinates
we may suppose that $E(g_1\mid B)$ does not depend on the
$A_1^c$-coordinates as well.
 By the previous theorem,
$$\int_{\xok} g_1g_2g_3\dots 
g_m\,d\mu_{[k]}=\int_{\xo^{A_1^c}}\left(\int_{\xo^{A_1}} 
g_1(x) g_2(x,y)\dots g_m(x,y)\,d\mu_{A_1}(x)\right)d\mu_{A_1^c}(y)\, .$$
Now we obtain by Lemma \ref{vetites} that for all $y\in X_{{A_1}^c}$ the
function $$x\to g_2(x,y)g_3(x,y)\dots g_m(x,y)~~(x\in X_{A_1})$$ is 
$B$-measurable.
This means that
$$\int_{\xo^{A_1}}g_1(x)g_2(x,y)\dots g_m(x,y) d\mu_{A_1}(x)=$$
$$=\int_{\xo^{A_1}}E(g_1|B)(x)g_2(x,y)g_3(x,y)\dots g_m(x,y) d\mu_{A_1}(x)$$ 
for all $y$ in $\xo_{A_1^c}$. 
This completes the proof. \qed \end{proof} \vskip 0.2in

\begin{lemma}[Total Independence]\label{totalindep} Let $A_1,A_2,\dots A_r$ be
  the list of nonempty
  subsets of $[k]$, and let $S_1,S_2,\dots,S_r$ be subsets of $\xok$ such
  that $S_i\in\sigma(A_i)$ and $E(S_i|\sigma(A_i)^*)$ is a constant 
function for
  every $1\leq i\leq r$. Then 
$$\mu(S_1\cap S_2\cap\dots\cap S_r)=\mu(S_1)\mu(S_2)\dots\mu(S_r).$$
\end{lemma}

\begin{proof} We can assume that $|A_i|\geq |A_j|$ whenever $j>i$. Let
  $\chi_i$ be the characteristic function of $S_i$. We have that
$$\mu(S_1\cap S_2\cap\dots\cap S_r)=\int_{\xok}\chi_1\chi_2\dots\chi_{r}
d\mu_{[k]}\,.$$
The integration rule shows that
$$\int_{\xok}\chi_i\chi_{i+1}\dots\chi_r\,d\mu_{[k]}=
\int_{\xok}E(\chi_i|\sigma(A_i)^*)\chi_{i+1}\dots\chi_r\,d\mu_{[k]}$$
$$=
\mu(S_i)\int_{\xok}
\chi_{i+1}\chi_{i+2}\dots\chi_r\,d\mu_{[k]}.$$
This completes the proof. \qed
\end{proof}

\subsection{Random Partitions}\label{randpar}
The goal of this section is to prove
the following proposition.
\begin{proposition}
\label{random}
Let $A\subset [k]$ be a subset, then for any $n\geq 1$ there exists
a partition $X_A=S_1\cup S_2\cup\dots\cup S_n$, such that
$E(S_i\mid \sigma(A)^*)=\frac{1}{n}$.
\end{proposition}
\begin{proof}
The idea of the proof is that we consider random partitions of $X_A$ and show
that by probability one these partitions shall satisfy the property of our
proposition.
Let $\Omega=\prod^\infty_{i=1}\{1,2,\dots,n\}^{X_{i,A}}$ be the
set of $\{1,2,\dots,n\}$-valued functions on $\cup^\infty_{i=1} X_{i,A}$.
Each element $f$ of $\Omega$ defines a partition of $X_A$ the following
way. Let
$$ S_f^{i,j}=\{p\in X_{i,A}\,\mid f(p)=j\}\,\,\,1\leq j \leq n,\, i\geq 1\,.$$
$$[\{S^{i,j}_f\}^\infty_{i=1}]= S^j_f\,.$$
Then $X_A=S^1_f\cup S^2_f\cup\dots \cup S^n_f$ is our partition induced by $f$.

\noindent
Note that on $\Omega$ one has the usual Bernoulli probability measure $P$,
$$P(T_{p_1,p_2,\dots,p_r}(i_1,i_2,\dots,i_r))=\frac{1}{n^r}\,,$$
where 
$$T_{p_1,p_2,\dots,p_r}(i_1,i_2,\dots,i_r)=\{f\in\Omega\,\mid\, f(p_s)=i_s\,
\,1\leq s \leq r\}\,.$$
A {\bf cylindric intersection set} $T$ in $X_{i,A}$
is a set $T=\cap_{C,C\subsetneq A} T_C$, where  $T_C\subset X_{i,C}$.
First of all note that the number of different dylindric intersection
sets in $X_{i,A}$ is not greater than
$$\prod_{C,C\subsetneq A} 2^{|X_{i,C}|}\leq 2^{(|X_i|^{A-1})2^k}\,.$$
Let $0\leq \e \leq\frac{1}{10n}$ be a real number and $T$ be a cylindric
intersection set of elements at least $\e |X_{i,A}|\,.$ By the
Chernoff-inequality the probability that an $f\in \Omega$ takes the
value $1$ more than $(\frac{1}{n}+\e)|T|$-times or less than
$(\frac{1}{n}-\e)|T|$-times on the set $T$ is less than
$2\exp(- c_{\e}|T|)$, where the positive constant $c_{\e}$ depends only
on $\e$. Therefore the probability that there exists a cylindric intersection
set $T\subset X_{i,A}$ of size at least $\e |X_{i,A}|$ for which 
$f\in \Omega$ takes the
value $1$ more than $(\frac{1}{n}+\e)|T|$-times or less than
$(\frac{1}{n}-\e)|T|$-times on the set $T$ is less than
$$2^{(|X_i|^{A-1})2^k} 2 \exp(-c_{\e}\e |X_i|^A)\,.$$
Since $|X_1|<|X_2|<\dots$
 by the Borel-Cantelli lemma we have the following lemma.
\begin{lemma}
For almost all $f\in\Omega$ there exist only finitely many $i$
such that there exists at least one cylindric intersection set
$T\subset X_{i,A}$ for which $f\in \Omega$ takes the
value $1$ more than $(\frac{1}{n}+\e)|T|$-times or less than
$(\frac{1}{n}-\e)|T|$-times on the set $T$.
\end{lemma}
Now let us consider a cylindric intersection set $Z\subseteq X_{A}$,
$Z=\cap_{C,C\subsetneq A} Z_C, \, Z_C\in X_C$.
By the previous lemma, for almost all $f\in \Omega$,
$$\mu(S^1_f\cap Z)=\frac{1}{n}\mu(Z)\,.$$
Therefore for almost all $f\in\Omega$:
$$\mu(S^1_f\cap Z')=\frac{1}{n}(\mu(Z'))\,,$$
where $Z'$ is a finite disjoint union of cylindric intersection sets
in $\xo_{A}$. Consequently, for almost all $f\in\Omega$,
$$\mu(S^1_f\cap Y)=\frac{1}{n}(\mu(Y))\,,$$
where $Y\in \sigma(A)^*$. This shows immediately that
$E(S^1_f\mid \sigma(A)^*)=\frac{1}{n}$ for almost all $f\in\Omega$. Similarly,
$E(S^i_f\mid \sigma(A)^*)=\frac{1}{n}$ for almost all $f\in\Omega$, thus our 
proposition follows.
\qed \end{proof} \vskip 0.2in

\subsection{Independent Complement in Separable $\sigma$-algebras}

Let $\mathcal{A}$ be a separable $\sigma$-algebra on a set $X$, and
let $\mu$ be a probability measure on $\mathcal{A}$. Two sub
$\sigma$-algebras $\mathcal{B}$ and $\mathcal{C}$ are called
independent if $\mu(B\cap C)=\mu(B)\mu(C)$ for every
$B\in\mathcal{B}$ and $C\in \mathcal{C}$. We say that $\mathcal{C}$
is an {\it independent complement} of $\mathcal{B}$ in $\mathcal{A}$
if it is independent from $\mathcal{B}$ and
$\langle\mathcal{B},\mathcal{C}\rangle$ is dense in $\mathcal{A}$.

\begin{definition} Let $\mathcal{A}\geq\mathcal{B}$ be two
  $\sigma$-algebras on a set $X$ and let $\mu$ be a probability measure on
  $\mathcal{A}$.
 A $\mathcal{B}$-random $k$-partition in $\mathcal{A}$ is
  a partition $A_1,A_2,\dots,A_k$ of $X$ into $\mathcal{A}$-measurable sets
  such that $E(A_i|\mathcal{B})=1/k$ for every $i=1,2,\dots,k$.
\end{definition}

\begin{theorem}[Independent Complement]\label{incom}
 Let $\mathcal{A}\geq\mathcal{B}$ be two separable $\sigma$-algebras on a set
  $X$ and let $\mu$ be a probability measure on $\mathcal{A}$. Assume that for
  every natural number $k$ there exists a $\mathcal{B}$-random $k$-partition
  $\{A_{1,k},A_{2,k},\dots,A_{k,k}\}$ in $\mathcal{A}$. Then there is
  an independent complement $\mathcal{C}$ of $\mathcal{B}$ in $\mathcal{A}$.
 (Note that this is basically the Maharam-lemma, see \cite{Mah})
\end{theorem}

\proof
Let $S_1,S_2,\dots$ be a countable generating system of $\mathcal{A}$ and
  let $\mathcal{P}_k$ denote the finite Boolean
 algebra generated by $S_1,S_2,\dots,S_k$ and
  $\{A_{i,j}|i\leq j\leq k\}$. Let $\mathcal{P}_k^*$ denote the atoms of
  $\mathcal{P}_k$. It is clear that for every atom $R\in\mathcal{P}_k^*$ we
  have that $E(R|\mathcal{B})\leq 1/k$ because $R$ is contained in one of the
  sets $A_{1,k},A_{2,k},\dots,A_{k,k}$. During the proof we fix one
  $\mathcal{B}$-measurable version of $E(R|\mathcal{B})$ for every $R$. The
  algebra $\mathcal{P}_{k}$ is a subalgebra of $\mathcal{P}_{k+1}$ for every
  $k$ and so we can define total orderings on the sets $\mathcal{P}_k^*$ such
  that if $R_1,R_2\in\mathcal{P}_k^*$ with $R_1<R_2$ and
  $R_3,R_4\in\mathcal{P}_{k+1}^*$ with $R_3\subseteq R_1, R_4\subseteq R_2$
  then $R_3<R_4$. We can assume that
  $\sum_{R\in\mathcal{P}_k^*}E(R,\mathcal{B})(x)=1$ for every element in
  $X$. It follows that for $k\in\mathbb{N}$, $x\in X$ and
  $\lambda\in [0,1)$ there is a
unique element $R(x,\lambda,k)\in\mathcal{P}_k^*$
  satisfying
$$\sum_{R<R(x,\lambda,k)}E(R|\mathcal{B})(x)\leq \lambda$$
and
$$\sum_{R\leq R(x,\lambda,k)}E(R|\mathcal{B})(x)>\lambda.$$
For an element $R\in\mathcal{P}_k^*$ let $T(R,\lambda,k)$ denote the
set of those points $x\in X$ for which $R(x,\lambda,k)=R$. It is
easy to see that $T(R,\lambda,k)$ is $\mathcal{B}$-measurable. Let
us define the $\mathcal{A}$-measurable set $S(\lambda,k)$ by
$$S(\lambda,k)=\bigcup_{R\in\mathcal{P}_k^*}(T(R,\lambda,k)
\cap(\cup_{R_2<R}R_2))$$
and $S'(\lambda,k)$ by
$$S'(\lambda,k)=\bigcup_{R\in\mathcal{P}_k^*}(T(R,\lambda,k)\cap(\cup_{R_2\leq
  R}R_2)).$$
\begin{proposition} \label{propindep}
\begin{description}
\item[(i)] $\lambda-\frac{1}{k}\leq 
E(S(\lambda,k)\mid \cB)(x)\leq \lambda$ for any $x\in X$.
\item[(ii)] If $k<t$, then
$S(\lambda,k)\subseteq S(\lambda,t)\subseteq S'(\lambda,k)\,.$
\item[(iii)] $E(S'(\lambda,k)\backslash S(\lambda,k)\mid\cB)(x)\leq
  \frac{1}{k} $ for any $x\in X$.
\end{description}
\end{proposition}
\proof
First observe that
$$ \lambda-\frac{1}{k}\leq \sum_{R< R(x,\lambda,k)}E(R\mid \cB)(x)\leq
\lambda\,,$$
for any $x\in X$. Also, we have
\begin{equation} \label{apr30}
S(\lambda,k)=\bigcup_{R,R_1\in \mathcal{P}_k^*, R<R_1}
(R\cap T(R_1,\lambda,k)),\quad
S'(\lambda,k)=\bigcup_{R,R_1\in \mathcal{P}_k^*, R\leq R_1}
(R\cap T(R_1,\lambda,k)). \end{equation}
That is by the basic property of the conditional expectation:
$$E(S(\lambda,k)\mid \cB)=
\sum_{R,R_1\in \mathcal{P}_k^*, R<R_1} E(R\mid \cB)
\chi_{T(R_1,\lambda,k)}\,.$$
That is
\begin{equation} \label{egy5}
E(S(\lambda,k)\mid \cB)(x)=\sum_{R<R(x,\lambda,k)}  E(R\mid \cB)(x)\,.
\end{equation}
and similarly
\begin{equation} \label{egy5b}
E(S'(\lambda,k)\mid \cB)(x)=\sum_{R\leq R(x,\lambda,k)}  E(R\mid \cB)(x)\,.
\end{equation}
Hence (i) and (iii) follows immediately, using the fact that
$E(R'\mid\cB)\leq \frac{1}{k}$ for any $R'\in \mathcal{P}_k^*$.

\noindent
Observe that for any $R\in \mathcal{P}_k^*$,
$T(R,\lambda,k)=\cup_{R'\subseteq R, R'\in \mathcal{P}_t^*}
T(R',\lambda,t)\,.$
Hence
$$\bigcup_{R,R_1\in \mathcal{P}_k^*, R<R_1}
(R\cap T(R_1,\lambda,k)) \subseteq
\bigcup_{R',R'_1\in \mathcal{P}_t^*, R'<R'_1}
(R'\cap T(R'_1,\lambda,t)) \subseteq$$ $$\subseteq
\bigcup_{R,R_1\in \mathcal{P}_k^*, R\leq R_1}
(R\cap T(R_1,\lambda,k))$$
Thus (\ref{apr30}) implies (ii)\,. 
 \qed

\vskip 0.2in
\noindent
\begin{lemma} Let $S(\lambda)=\cup^\infty_{k=1} S(\lambda,k)\,.$
 Then if $\lambda_2<\lambda_1$, then $S(\lambda_2)< S(\lambda_1)$.
\end{lemma}
\proof Note that $x\in S(\lambda_2,k)$ if and only if
$x\in R_2$ for some $R_2< R(x,\lambda_2,k)\,.$ Obviously,
$R(x,\lambda_2,k)< R(x,\lambda_1,k)$, thus $x\in S(\lambda_1,k)$. Hence
$S(\lambda_2)\subseteq S(\lambda_1)$ \qed
\begin{lemma} $E(S(\lambda)\mid \cB)=\lambda$.
\end{lemma}
\proof
Since $\chi_{S(\lambda,k)}\stackrel
{L_2(X,\mu)}{\to} \chi_{S(\lambda)}$, we have
$E(S(\lambda,k)\mid \cB)
\stackrel{L_2(X,\mu)}{\to}E(S(\lambda)\mid \cB)$ That is
by (i) of Proposition \ref{propindep} $E(S(\lambda)\mid \cB)=\lambda$. \qed
\vskip 0.2in
\noindent
The last two lemmas together imply that the sets $S(\lambda)$ generate a
$\sigma$-algebra $\mathcal{C}$ which is independent from $\mathcal{B}$.

\noindent
Now we have to show that $\mathcal{B}$ and $\mathcal{C}$ generate
$\mathcal{A}$. Let $S\in\mathcal{P}_k$  for some
$k\in\mathbb{N}$. We say that $S$ is an interval if there exists an element
$R\in\mathcal{P}_k^*$ such that $S=\cup_{R_1\leq R}R_1$.  It is enough to show
that any interval $S\in\mathcal{P}_k$ can be generated by $\mathcal{B}$ and
$\mathcal{C}$.

\noindent
Suppose that $\{T_t\}^\infty_{t=1}$ be sets in $\langle\cB,\cC\rangle$ and
$\|E(S\mid \cB)-E(T_t\mid \cB)\|^2\to 0$. Then $\mu(S \triangle T_t)\to 0$
as $t\to 0$, that is $\cB$ and $\cC$ generate $S$.
Indeed,
$$\mu(S\triangle T_n)^2=\| \chi_S- \chi_{T_n}\|^2\geq \|E(S\mid B)- E(T_n\mid
B)\|^2\,.$$
So let $t\geq k$ be an arbitrary natural number. It is clear that
$S$ is an interval in $\mathcal{P}_t$. For a natural number $0\leq d\leq t-1$
let $F_d$ denote the $\mathcal{B}$-measurable set on which $E(S|\mathcal{B})$
is in the interval $[\frac{d}{t},\frac{d+1}{t})$. Now we approximate $S$ by
$$T_t=\bigcup_{d=0}^{t-1}(F_d\cap S(\frac{d}{t}))\in
\langle\mathcal{B},\mathcal{C}\rangle.$$
\begin{lemma} For any $x\in X$,
$$\left|E(S\mid \cB)(x)-E(T_t\mid \cB)(x)\right|\leq \frac{3}{t}\,.$$
\end{lemma}
\proof
First note that by Proposition \ref{propindep} (iii)
$$\left|E(S(\frac{d}{t})\mid \cB)(x)- E(S(\frac{d}{t},t)\mid \cB)(x)\right|\leq
\frac{1}{t}\,. $$
Note that
$$E(T_t\mid \cB)(x)=\sum^{t-1}_{d=0} \chi_{F_d}(x)
E(S(\frac{d}{t})\mid\cB)(x)\,.$$
Suppose that $x\in F_d$. Then
$$\left|E(T_t\mid \cB)(x)-\sum_{R'<R(x,\frac{d}{t},t)} 
E(R'\mid \cB)(x)\right|\leq 
\frac{1}{t}\,. $$
On the other hand
$E(S\mid\cB)(x)=\sum_{R'\leq R} E(R'\mid \cB)(x)$ and
$\frac{d}{t}\leq \sum_{R'\leq R} E(R'\mid \cB)(x) < \frac{d+1}{t}\,.$
That is 
$$\left|E(S\mid \cB)(x)-E(T_t\mid \cB)(x)\right|\leq \frac{3}{t}
\,.\quad \qed $$
The Theorem now follows from the Lemma immediately. \qed

\begin{definition} Let $(X,\mathcal{A},\mu)$ be a probability space, and
  assume that a finite group $G$ is acting on $X$ such that $\mathcal{A}$ is
  $G$-invariant as a set system. We say that the action of $G$ is free if
  there is a subset $S$ of $X$ with $\mu(S)=1/|G|$ such that $S^{g_1}\cap
  S^{g_2}=\emptyset$ whenever $g_1$ and $g_2$ are distinct elements of $G$.
\end{definition}

We will need the following consequence of Theorem \ref{incom}.

\begin{lemma}\label{incom2} Let $\mathcal{A}\geq\mathcal{B}$ be two separable
  $\sigma$-algebras on the set $X$ and let $\mu$ be a probability measure on
  $\mathcal{A}$. Assume that a finite group $G$ is acting on $X$ such that
  $\mathcal{A},\mathcal{B}$ and $\mu$ are $G$ invariant. Assume furthermore
  that the action of $G$ on $(X,\mathcal{B},\mu)$ is free and that there is a
  $\mathcal{B}$-random $k$ partition of $X$ in $\mathcal{A}$ for every natural
  number $k$. Then there is an independent complement $\mathcal{C}$ in
  $\mathcal{A}$ for $\mathcal{B}$ such that $\mathcal{C}$ is 
  elementwise $G$-invariant.
\end{lemma}

\begin{proof} Let $S\in\mathcal{B}$ be a set showing that $G$ acts freely
  on $\mathcal{B}$. Let $\mathcal{A}|_S$ and
  $\mathcal{B}|_S$ denote the restriction of $\mathcal{A}$ and
  $\mathcal{B}$ to the set $S$. It is clear that if
  $\{A_1,A_2,\dots,A_k\}$ is a $\mathcal{B}$-random $k$-partition in
  $\mathcal{A}$ then $\{S\cap A_1,S\cap A_2,\dots,S\cap A_k\}$ is a
  $\mathcal{B}|_S$-random $k$ partition in $\mathcal{A}|_S$. Hence by
  Theorem \ref{incom} there exists an independent complement $\mathcal{C}_1$ of
  $\mathcal{B}|_S$ in $\mathcal{A}|_S$. The set
$$\mathcal{C}=\{\bigcup_{g\in G} H^g|H\in\mathcal{C}_1\}$$ is a
$\sigma$-algebra because the action of $G$ is free. Note that the
elements of $\mathcal{C}$ are $G$-invariant. Since $E(\cup_{g\in G}
H^g|\mathcal{B})=\sum_{g\in G} E(H|\mathcal{B}|_S)^g$ we get that
the elements of $\mathcal{C}$ are independent form $\mathcal{B}$. It
is clear that $\langle\mathcal{C},\mathcal{B}\rangle$ is dense in
$\mathcal{A}$. \qed
\end{proof}

\subsection{Separable Realization}\label{sepreal}

In this section we show how to pass from nonseparable $\sigma$-algebras to
separable ones.

First note that the symmetric group $S_k$ acts on the space $X^k$ by
permuting the coordinates:
$$(x_1,x_2,\dots,x_k)^\pi=(x_{\pi^{-1}(1)}, x_{\pi^{-1}(2)},\dots,
x_{\pi^{-1}(k)})\,.$$
The group also acts on the subsets of $[k]$ and $\sigma(A)^\pi=
\sigma(A^\pi)$, where $A^\pi$ denotes the image of the subset $A$
under $\pi\in S_k$. We shall denote by $S_A$ the symmetric group
acting on the subset $S_A$.

\begin{definition}\label{sepre} A {\bf separable realization} of degree $r$ on
  $\xo^k~~, r\leq k$ is a system of atomless separable $\sigma$-algebras
  $\{l(A)~|~\emptyset\neq A\subseteq [k]~,~|A|\leq r\}$ and 
functions $\{F_A:\xo^k\to
  [0,1]~|~\emptyset\neq A\subseteq [k]~,~|A|\leq r\}$ with the 
following properties
\begin{enumerate}
\item $l(A)$ is a subset of $\sigma(A)$ and is independent 
from $\sigma(A)^*$ for every $\emptyset\neq
  A\subseteq [k]$.
\item $l(A)^\pi=l(A^\pi)$ for every permutation $\pi\in S_k$.
\item $S^\pi=S$ for every $S\in l(A)$ and $\pi\in S_A$.
\item $F_A$ is an $l(A)$-measurable function which defines a measurable 
equivalence
between the measure algebras of $(\xo^k,l(A),\muo^k)$ and $[0,1]$. (see
Appendix)
\item $F_A(x)=F_{A^\pi}(x^\pi)$ for every element 
$x\in\xo^k~,~\pi\in S_k$ and
$A\subseteq [k]$.
\end{enumerate}
\end{definition}

 The main theorem in this section is the following one.

\begin{theorem}\label{reali} For every $\bH\in\sigma([k])$ there exists
 a separable realization
  of degree $k$ such that $\bH$ is measurable in $\langle l(A)~|~\emptyset\neq
  A\subseteq [k]\rangle$.
\end{theorem}

We will need the following three lemmas.

\begin{lemma}\label{sep1} Let $\mathcal{B}\subseteq\mathcal{A}$ two
$\sigma$-algebras on a
  set $X$, and let $\mu$ be a probability measure on $\mathcal{A}$. Then for
  any separable sub-$\sigma$-algebra $\bar{\mathcal{A}}$ of $\mathcal{A}$
  there exists
 a separable sub $\sigma$-algebra $\bar{\mathcal{B}}$ of $\mathcal{B}$
  such that $E(A|\mathcal{B})=E(A|\bar{\mathcal{B}})$
for every $A\in\bar{\mathcal{A}}$.
\end{lemma}

\proof
  We use the fact that $\bar{\mathcal{A}}$ is a
  separable metric space with the distance $d(A,B)=\mu(A\triangle B)$. Let
  $W=\{D_1,D_2,\dots\}$ be a countable dense subset of $\bar{\mathcal{A}}$
  with the previous distance. Let $C_{p,q}^i=E(D_i\mid \cB)^{-1}(p,q)$,
where $p<q$ are rational numbers. Clearly, $E(D_i\mid \cB)$ is a
$\cB_i$-measurable function, where $\cB_i=\langle C_{p,q}^i\mid
p<q\in \bQ\rangle$. Obviously, $E(D_i\mid \overline{\cB})=
E(D_i\mid \cB)$ for any $i\geq 1$, where
$\overline{\cB}=\langle \cB_i\mid i=1,2,\dots\rangle\,.$
Now observe that $E(D_i\mid \overline{\cB})\stackrel {L_2}{\to}
 E(D,\overline{\cB})$
if $D_i\to D$. Hence for any $D\in \overline{\cA}$, $E(D\mid\overline{\cB})=
E(D\mid\cB)$.\qed
\begin{lemma}\label{inv1} Let $A\subseteq[k]$ be a subset and assume that
  there are atomless separable $\sigma$-algebras 
$d(\{i\})\subset\sigma(\{i\})$\,,$i\in A$
  such that $d(\{i\})^\pi=d(\{i^\pi\})$ for every $i\in A$ and $\pi\in
  S_A$. Then $S_A$ acts freely on $\langle d(\{i\})|i\in A\rangle$.
\end{lemma}

\proof The permutation invariance implies that there is a
  $\sigma$-algebra $\mathcal{A}$ on $X$ such that
  $P_{\{i\}}^{-1}(\mathcal{A})=d(\{i\})$ for every $i\in A$.
  Let $F:\xo\to [0,1]$ be a $\mathcal{A}$-measurable measure 
preserving map. Now
  we can define the map $G:\xo^A\to [0,1]^{A}$ by
  $$G(x_{i_1},x_{i_2},\dots,x_{i_{|A|}}):=
(F(x_{i_1}),F(x_{i_2}),\dots,F(x_{i_{|A|}})).$$
Let us introduce $S':=\{(x_1,x_2,\dots,x_r)|x_1<x_2<\dots< x_r\}\subset[0,1]^A$
and $S:=G^{-1}(S')$. Clearly $\muo^A(S)=1/|A|!$ 
and $S^\pi\cap S^\rho=\emptyset$
for every two different elements $\pi\neq\rho$ in $S_A$.
\qed

\begin{lemma}\label{sep2} Let $k$ be a natural number and assume that for every
  $A\subseteq[k]$ there is a separable $\sigma$-algebra $c(A)$ in
  $\sigma(A)$. Then for every $A\subseteq[k]$ there is a separable
  $\sigma$-algebra $d(A)$ in $\sigma(A)$ with $c(A)\subseteq d(A)$ such that

\begin{enumerate}
\item  $E(R|\langle d(B)|B\in A^*\rangle)=E(R|\sigma(A)^*)$ whenever $R\in
  d(A)$.
\item $d(A)^\pi=d(A^\pi)$ for every element $\pi\in S_k$.
\item $d(B)\subseteq d(A)$ whenever $B\subseteq A$
\end{enumerate}
\end{lemma}

\proof First we construct algebras $d'(A)$
  recursively. Let $d'([k])$ be
  $\langle c([k])^\pi|\pi\in S_k\rangle$.
Assume that we have already constructed the algebras $d'(A)$ for
  $|A|\geq t$. Let $A\subseteq[k]$ be such that $|A|=t$. By Lemma \ref{sep1}
  we can see that there exists a separable subalgebra $d'(A)^*$ of
$\sigma(A)^*$ such that
  $E(R|\sigma(A)^*)=E(R|d'(A)^*)$ for every $R\in d'(A)$. 
Since $\sigma(A)^*$ is
  generated by the algebras $\{\sigma(B)|B\in A^*\}$ we have that
  every element of $\sigma(A)^*$ is a countable expression of some
  sets in these algebras. This implies that any separable sub $\sigma$-algebra
  of $\sigma(A)^*$ is generated by separable sub
  $\sigma$-algebras of the algebras $\sigma(B)$ where $B\in A^*$. 
In particular we can choose
 separable $\sigma$-algebras $d'(A,B)\supset c(B)$ in
$\sigma(B)$ for every $B\in A^*$ such
 that $\langle d'(A,B)|B\in A^*\rangle\supseteq d(A)^*$. For a set
 $B\subseteq[k]$ with $|B|=t-1$ we define $d'(B)$ as the $\sigma$-algebra
 generated by all the algebras in the form of $d'(C,D)^\pi$,
where $\pi\in S_k$ , $D^\pi=B$ ,
 $|C|=|D|+1$ and $D\subseteq C$. Since
 $d'(C,D)^\pi\subseteq\sigma(D)^\pi=\sigma(B)$ we have that
 $d'(B)\subseteq\sigma(B)$. Furthermore we have that
 $d'(B)^\pi$=$d'(B^\pi)$ for every $\pi\in S_k$.

Now let $d(A):=\langle d'(B)~|~B\subseteq A\rangle$. the second
requirement in the lemma is trivial by definition. We prove the
first one. The elements of $d(A)$ can be approximated by finite
unions of intersections of the form $\bigcap_{B\subseteq A}T_B$
where $T_B\in d'(B)$ and so it is enough to prove the statement if
$R$ is such an intersection. Let $Q=\bigcap_{B\subset A,B\neq
A}T_B$. Now

$$E(R|\langle d(B)|B\in A^*\rangle)=E(R|\langle d'(B)|B\subset
A,B\neq A\rangle)\,.$$ By the basic property of the conditional expectation
(see Appendix) :
$$E(R|\langle d'(B)|B\subset
A,B\neq A\rangle)=E(T_A|\langle d'(B)|B\subset A,B\neq A\rangle)
\chi_Q=E(T_A|\sigma(A)^*)\chi_Q=$$$$=E(R|\sigma(A)^*).$$
\qed

\medskip
\noindent{\bf Proof of Theorem \ref{reali}}~~ We construct the
algebras $l(A)$ in the following steps. For each non-empty subset
$A\subseteq[k]$ we choose an atomless separable $\sigma$-algebra
$c(A)\subseteq\sigma(A)$ containing a $\sigma(A)^*$-random
$r$-partition for every $r$. We also assume that $\bH\in c([k])$. Applying 
Lemma \ref{sep2} for the previous
system of separable
  $\sigma$-algebras $c(A)$ we obtain the $\sigma$-algebras $d(A)$. 
By Lemma \ref{inv1} and the permutation invariance property of the
previous lemma, $S_{[r]}$ acts freely on
  $d([r])^*$. Hence
using Lemma \ref{incom2}, for every $\emptyset\neq A\in[k]$ we can choose an 
independent
 complement $l([r])$ for $d([r])^*=\langle d(B)|B\in
  [r]^*\rangle$ in $d([r])$ such 
that $l([r])$ is element-wise invariant under the
  action of $S_{[r]}$.  The algebras $l([r])$ are independent from 
$\sigma([r])^*$
since $\mu(R)=E(R|d([r])^*)=E(R|\sigma([r])^*)$ for every $R\in l([r])$. 
Now we define $l(A)$, where $|A|=r$ by $l(A)=l([r])^\pi$ for some
$\pi \in S_k$, $\pi([r])=A$. Note that $l(A)$ does not depend on the choice
of $\pi$. By Lemma \ref{measurealgebra} of the Appendix we have
maps $F_{[r]}:X\to [0,1]$ such that $F^{-1}$ defines a measure
algebra isomorphism between $\cM([0,1],\cB,\lambda)$ and $\cM(\xo,l[r],\mu)$.
Let $F_A=\pi^{-1}\circ F_{[r]}$, where $\pi$ maps $[r]$ to $A$. Again,
$F_{[r]}$ does not depend on the particular choice of the
permutation $\pi$. \qed

\vskip 0.2in
\noindent
Now let $\mathcal{S}=\{l(A),F_A\}_{\emptyset\neq A\subseteq [k]}$ be
a separable realization of $\xo$ and $k<n$ be a natural number.
Let $B\subset [n], |B|=r\leq k$ and $\pi\in S_n$ be a permutation that
maps $[r]$ to $B$. Let $l(B)\subseteq \xo^n$ be defined as
$l([r])^{\pi}$. If we choose a $\pi'\in S_n$ that also maps
$[r]$ into $B$ then $(\pi)^{-1}\circ \pi$ permutes $[r]$ hence fixes
$l([r])$. Therefore $l(B)$ does not depend on the choice of $\pi$.
Let $F_B$ defined as $\pi^{-1}\circ
 F_{[r]}\,.$
We have the following lemma.
\begin{lemma} 
The system $\widetilde{\mathcal{S}}=\langle l(B), F_B \rangle_{\emptyset\neq
  B\subseteq [n],|B|\leq k  }$ is a 
separable realization of degree $k$ on $X^n$.
If $B\subset [n], |B|=r\leq k$ and $f:[r]\to B$ is a bijection then
let $p_B:\xo^n\to \xo^B$ the natural projection and $p_f: \xo^{[r]}\to \xo^B$
is the natural isomorphism. Then $l(B)=p_B^{-1}\left(p_f ( l([r]))\right)$.
\end{lemma}
By Lemma \ref{totalindep} and Lemma \ref{fremlin} we have
the following lifting lemma as well.
\begin{lemma}
\label{lifting}
The map $F:X^k\to [0,1]^{2^k-1}$, $F=\oplus_{\emptyset\neq A\subseteq [k]}F_A$
defines an 
isomorphism between the measure algebra of
 $\cM(\xo^k,\langle l(A)\mid\emptyset\neq A\subseteq [k]\rangle,\mu)$ and 
the Lebesgue measure
algebra $\cM([0,1]^{2^k-1},\cB,\lambda)$. Similarly, $\widetilde{F}:\xo^n\to 
[0,1]^{\sum^k_{i=1}(\stackrel{n}{i})}$, $\widetilde{F}=
\oplus_{\emptyset\neq B\subseteq [n], |B|\leq k}F_B$
defines an 
isomorphism between the measure algebra 
 $\cM(\xo^n,\langle l(B)\mid\emptyset\neq 
B\subseteq [n],|B|\leq k \rangle,\mu)  $ and 
the Lebesgue measure
algebra $\cM([0,1]^{\sum^k_{i=1}(\stackrel{n}{i})},\cB,\lambda)$. \end{lemma}
\section{Applications for Hypergraphs}
\subsection{Hypergraph homomorphisms and convergence}
Recall that a $k$-uniform {\bf hypergraph} $H$ is a system of $k$ element
subsets ({\bf edges}) denoted by $E(H)$ of a set $V$ ({\bf node set}).
 A $k$-uniform hypergraph can be represented as 
a subset $S_H\subset V^k$ such that
$(x_1,x_2,\dots x_k)\in S$ if and only if
$\{x_1, x_2, \dots, x_k\}\in E(H)$. 
Note that $S_H$ is invariant under the
action of $S_k$ on $V^k$. For any hypergraph
 we have an underlying $(k-1)$-dimensional simplicial
complex $\Sigma(H)$ consisting of the subsets of the $k$-edges.

\noindent
Suppose that $K$ is a finite $k$-uniform hypergraph on the node set
$[n]:=\{1,2,\dots,n\}$ and $H$ is a $k$-uniform hypergraph on the node set
$V$. Then
a map
$f:[n]\to V$ is a {\bf homomorphism} if $f$ maps edges to edges.
If $H$ is finite then $hom(K,H)$ is the number of homomorphism from $K$
to $H$. Denote by $t(K,H)$ the probability that a random map
$g:[n]\to V$ is a $(K,H)-homomorphism$, that is
$$t(K,H)=\frac{hom(K,H)}{|V|^n}\,.$$
If $H$ is not necessarily finite then $T(K,H)\subset V^n$ denotes
the $(K,H)$-{\bf homomorphism set}, where
$(x_1,x_2,\dots,x_n)\in T(K,H)$ if
$1\to x_1, 2\to x_2,\dots, n\to x_n$ defines a homomorphism.
Clearly $|T(K,H)|=hom(K,H)$. Note that
$$T(K,H)=\bigcap_{E\in E(K)} p_E^{-1}\left(p_f (S_H)\right)\,,$$ where
$f:[k]\to B$ is a bijection (see Lemma \ref{lifting}).

\begin{definition}
We say that a sequence of $k$-uniform hypergraphs $\{H_i\}^\infty_{i=1}$
is {\bf convergent} if for every fixed finite $k$-uniform hypergraph $K$
$lim_{i\to\infty} t(K,H)$ exists.
\end{definition}
Let $\{X_i\}^\infty_{i=1}$ be finite sets and $H_i\subset X_i^k$ be
$k$-uniform directed hypergraphs, that is
a sequence of $S_k$-invariant sets
 $S_{H_i}\subset X_i^k$ is given. As in the Section \ref{sec1}
, let $X$ be the ultralimit of the sets $X_i$. Then $\bH:=
[\{S_{H_i}\}^\infty_{i=1}]\subset \cP(X^k)$ is the {\bf ultralimit
  hypergraph},
an $S_k$-invariant set corresponding to an actual hypergraph on the node set
$X$. We can define its homomorphism set as
$$T(K,\bH):=\bigcap_{E\in E(K)} p_E^{-1}\left( p_f (\bH)\right)\,.$$
Then
$$T(K,\bH)=[\{T(K,H_i)\}^\infty_{i=1}]\subset \cP(\xo^n)\,.$$
Clearly, $\muo^n(T(K,\bH))=\limo t(K,H_i)$, where
$\muo^n$ denotes the ultralimit measure on $\xo^n$. Thus if
$\{H_i\}^\infty_{i=1}$ is a convergent sequence of hypergraphs 
then:
$$\muo^n(T(K,\bH))=
\lim_{i\to\infty} t(K,H_i)\,.$$

\subsection{The Hypergraph Removal Lemma}

\begin{lemma}[Infinite Removal Lemma]\label{infrem} 
Let $\bH$ be an $S_k$-invariant measurable subset of
  $\xo^k$. Then there exists an $S_k$-invariant nullset $\bI\subseteq\bH$ such
 that for every
  $k$-uniform hypergraph $K$ either $T(K,\bH\setminus \bI)=\emptyset$ or
  $|T(K,\bH\setminus \bI)|>0$.
\end{lemma}

\proof
 Let us consider the separable realization $\mathcal{S}$ of $\bH$ and the
  corresponding measurable equivalence $F:\xo\to [0,1]^{2^k-1}$. For some
  Lebesgue measurable set $Q\subseteq [0,1]^{2^k-1}$ we have that 
$|F^{-1}(Q)\triangle\bH|=0$. Since
$$F^{-1}(Q^\pi)\triangle \bH^\pi=(F^{-1}(Q)\triangle \bH)^\pi\,$$
we may suppose that $Q$ is $S_k$-invariant.
 By Lebesgue's Density Theorem,
  almost all points of $Q$ are density points. 
Let $D$ denote the ($S_k$-invariant)
 set of density points in $Q$ and let $S:=F^{-1}(D)$. 
Notice that the group $S_k$ acts on $[0,1]^{2^k-1}$ the following way.
Let $A_1, A_2,\dots, A_{2^k-1}$ be a list of non-empty subsets of $[k]$. Then
$$(y_{A_1}, y_{A_2},\dots,y_{A_{2^k-1}})^\pi= (y_{\pi^{-1}(A_1)}, 
y_{\pi^{-1}(A_2)}, \dots, y_{\pi^{-1}(A_{2^k-1})})\,.$$
By the invariance property of the separable realization, the maps $F_A$
commutes with the $S_k$-action that is $\pi\circ F_A=F_A\circ \pi$.
Also, let $B_1, B_2,\dots, B_r\,\,,(r=\sum_{i=1}^k (\stackrel {n} {i}))$
be the list of non-empty subsets of $[n]$ 
of size at most $k$, then $S_n$ acts on $[0,1]^r$
by
$$(y_{B_1}, y_{B_2},\dots,y_{B_r})^\pi=
(y_{\pi^{-1}(B_1)}, 
y_{\pi^{-1}(B_2)}, \dots, y_{\pi^{-1}(B_r)})\,.$$
Again, by the invariance property of the lifting $\rho\circ F_B =
F_B\circ \rho$, for any $B\subset [n]$, $|B|\leq k$, $\rho\in S_n$.
For $B\subset [n]$, $|B|=k$ a bijection $f:[k]\to B$ induces a measurable
isomorpism $L_f:[0,1]^{2^k-1}\to [0,1]^{r(B)}$, where
$r(B)$ denotes the set of non-empty subsets of $B$.
Let $L_B:[0,1]^r\to [0,1]^{r(B)}$ be the natural projection. Then by the
invariance property of the lifting 
\begin{equation} \label{cilin}
p_B^{-1}(p_f(\xo^k))=\widetilde{F}^{-1}(L_B^{-1}(L_f
([0,1]^{2^k-1})))\,.
\end{equation}
That is for any $k$-regular hypergraph $K$
$$T(K,\bS)=\cap_{E\in E(K)} p_E^{-1}(p_f(\bS))=
\widetilde{F}^{-1}(\cap_{E\in E(K)} L_E^{-1}(L_f(D)))\,.$$
Since each point of $D$ is a density point, each point of
$L_E^{-1}(L_f (D))$ is a density point for any $E\in E(K)$.
Thus $\cap_{E\in E(K)} L_E^{-1}( L_f(D))$ is either empty
or is of positive measure.
Consequently,
$T(K,\bS)$ is either empty or is of positive measure as well.
Choosing $\bI=H\backslash \bS$, we obtain that
$T(K, H\backslash \bI)= T(K,H\cap \bS)$ is either empty or is of positive
measure (note that $\mu(T(K,H\cap \bS))=\mu(T(K,\bS))$ and 
$T(K,H\cap \bS)\subseteq
T(K,\bS))\,.$ \qed

\begin{theorem}[Hypergraph Removal Lemma] For every $k$-uniform hypergraph $K$
  and constant $\epsilon>0$ there exists
 a number $\delta=\delta(K,\epsilon)$ such
  that for any $k$-uniform hypergraph $H$ on the node set $X$ with
  $t(K,H)<\delta$ there is a subset $L$ of $E(H)$ with $L\leq
  \epsilon{{|X|}\choose{k}}$ such that $t(K,H\setminus L)=0$. (\cite{Gow}.
\cite{Ish}, \cite{NRS}, \cite{Tao})
\end{theorem}

\proof We proceed by contradiction. Let $K$ be a fixed hypergraph and
  $\epsilon>0$ be a fixed number for which the theorem fails. This means that
  there is a sequence of hypergraphs $H_i$ on the sets $X_i$ such that
  $lim_{i\to\infty}t(K,H_i)=0$ but in each $H_i$ there is no set $L$ with the
  required property. Let us represent the hypergraphs by symmetric subsets
 $S_{H_i}$ of
  $X_i^k$ and again
  let $\bH\subseteq\xo^k$ denote the ultralimit of them. Then
   $\mu(T(K,\bH))=\limo t(K,H_i)=0$ and thus by the previous lemma
  there is a
  zero measure $S_k$-invariant
  set $\bI\subseteq\xo^k$ such that $T(K,\bH\setminus
  \bI)=\emptyset$. By the definition of nullsets, for any $\e_1>0$ there exists
 an ultralimit set $\bJ\subset \xo^k$ such that $\bI\subset \bJ$ and 
$\mu(\bJ)\leq
 \e_1$. We can suppose that $\bJ$ is $S_k$-invariant as well. Let
 $[\{J_i\}_{i=1}^\infty]=\bJ$, then 
for $\omega$-almost all $i$, $J_i$ is $S_k$-
 invariant, $|J_i|\leq \e_1|X_i|^k$ and $T(K,H_i\backslash L_i)=\emptyset$,
 where $L_i$ is the set of edges $\{x_1,x_2,\dots, x_k\}$ such that
 $(x_1,x_2,\dots,x_k)\in J_i$. Clearly, $|L_i|\leq |J_i|$, hence if
$\e_1$ is small enough then $|L_i|\leq \epsilon{{|X|}\choose{k}}$ leading
to a contradiction. \qed
\subsection{The Hypergraph Limit Object}\label{hypgr}
In this section we introduce the notion of {\bf hypergraphons} (see
\cite{LSZ}
and \cite{Borgs} for {\bf graphons}).
Let $W:[0,1]^{2^k-1}\to \{0,1\}$ be a Lebesgue measurable function. We
call such functions {\bf directed hypergraphons}. As in the previous
subsection we consider the $S_k$-action on $[0,1]^{2^k-1}$ and
call the $S_k$-invariant directed hypergraphons just
hypergraphons\,. Now we introduce the homomorphism density of a
hypergraph into a hypergraphon. Let $K$ be a $k$-uniform
hypergraph and $W:[0,1]^{2^k-1}\to\bR$ be a
hypergraphon. Let $C_K=\{C_1, C_2,\dots, C_s\}$ be the set of non-empty
elements of the simplicial complex of $K$.

\noindent
{\bf Example:} If $K=\{\{1,2,3\},\{2,3,4\}\}$ then
$$ C_K=\{\{1\}, \{2\},  \{3\},  \{4\},  \{12\},  \{13\}, \{23\},
 \{24\},\{34\},\{1,2,3\},\{2,3,4\}\}.$$
For each edge $E\in E(K)$ we fix a bijection $s_E:[k]\to E$.
Then the {\bf homomorphism density} of $K$ in $W$
is defined as \begin{equation} \label{density}
 t(K,W):=\int_0^1\int_0^1\dots\int_0^1\,\prod_{E\in E(K)}
W_\bH(x_{s_E(A_1)}, x_{s_E(A_2)},\dots, x_{s_E(A_{2^k-1})}) dx_{C_1}
dx_{C_2}\dots dx_{C_s}\,. \end{equation}

%Note that in the special case of graphs (\ref{density}) is just the
%homomorphism density defined in \ref{LSZ}.

Now let $\{X_i\}^\infty_{i=1}$ be finite sets and $H_i$ be
$k$-uniform directed hypergraphs on $X_i$. Let $\bH\subset \xo^k$ be
their ultralimit hypergraph. Let $F:X^k\to [0,1]^{2^k-1}$ be separable
realization and $Q\subset [0,1]^{2^k-1}$ be a $S_k$-invariant measurable
set such that $\mu(F^{-1}(Q)\triangle \bH))=0$. Then we define
$W_\bH$ as the characteristic function of $Q$.
 Clearly, $W_\bH$ is an hypergraphon. Now we can state our main
theorem.

\begin{theorem}[Main Theorem] Let $\{H_i\}_{i=1}^{\infty}$ be a 
sequence of
  $k$-uniform hypergraphs as above and let $K$ be a fixed $k$-uniform
  hypergraph on the vertex set $[n]$. Then
$$\limo~ t(K,H_i)=t(K,W_\bH).$$
\end{theorem}

\proof 
Applying the Equation (\ref{cilin}) we obtain that
$$\mu(T(K,H)=Vol\left(\cap_{E\in E(K)} L^{-1}_E (L_f (Q))\right)\,.$$
Hence
$$\mu(T(K,H)=\int^1_0 \int^1_0 \dots \int^1_0
\prod_{E\in E(K)} \Psi_E  dx_{B_1}
dx_{B_2}\dots dx_{B_r}\,,$$
where $\Psi_E$ is the characteristic function of $L^{-1}_E (L_f(Q))$.
Clearly, 
$$\Psi_E(x_{B_1},x_{B_2},\dots, x_{B_r})=
W_\bH(x_{s_E(A_1)}, x_{s_E(A_2)},\dots, x_{s_E(A_{2^k-1})})\,.$$
Since $\prod_{E\in E(K)} \Psi_E$ depends only on the variables associated
to the elements of the simplicial complex of $K$, the Theorem follows.
\qed

\vskip 0.2in
\noindent

The following theorem is an immediate corollary of the previous one.

\begin{theorem} If $\{H_i\}_{i=1}^{\infty}$ is a convergent sequence of
  $k$-uniform hypergraphs then there exists a $2^k-1$ variable
  hypergraphon $W$ such that $lim_{i\to\infty}t(K,H_i)=t(K,W)$ for every
  $k$-uniform hypergraph $K$.
\end{theorem}

\medskip

\noindent
{\bf Remark:}
One can introduce the notion of a {\bf projected hypergraphon} $
\widetilde{W}_{\bH}$
 which
is the projection of a hypergraphon to the first $2^k-2$ coordinates, where
the last coordinate is associated to $[k]$ itself.
That is 
$$\widetilde{W}_{\bH}(x_{A_1}, x_{A_2}, \dots, x_{A_{2^k-2}})=
\int_0^1 W_{\bH} (x_{A_1}, x_{A_2}, \dots, x_{A_{2^k-1}}) dx_{A_{2^k-1}}\,.$$
That is $\widetilde{W}_{\bH}$ is a $[0,1]$-valued function which is
symmetric under the induced $S_k$-action of its coordinates.
By the classical Fubini-theorem we obtain that using the notation of the
previous theorem: 
$$lim_{i\to\infty}t(K,H_i)=$$
$$=\int_0^1\int_0^1\dots \int_0^1
\prod_{E\in E(K)}\widetilde{W}_{\bH} (x_{s_E(A_1)}, x_{s_E(A_2)},\dots, 
x_{s_E(A_{2^k-2})}) dx_{C_1}
dx_{C_2}\dots dx_{C_t}\,,$$
where $C_1, C_2, \dots, C_t$ is the list of the
at most $k-1$-dimensional simplices
in $K$. Note that in the case $k=2$ it is just the graph limit formula of
\cite{LSZ}.

\subsection{The Hypergraph Regularity Lemma}
 First we need some definitions. Let $X$ be a finite set, then
$K_r(X)$ denotes the complete $r$-uniform hypergraph on $X$.
An $l$-{\bf hyperpartition} $\cH$ is a family of partition
$K_r(X)=\cup^l_{j=1} P^j_r$, where $P^j_r$ 
is an $r$-uniform hypergraph, for $1\leq r \leq k$.
We call $\cH$ $\delta$-{\bf equitable}
if for any $1\leq r \leq k$ and $1\leq i< j \leq l$:
$$\frac{||P^i_r|-|P^j_r||}{|K_r(X)|} <\delta\,.$$
An $l$-hyperpartition $\cH$ induces a partition on $K_k(X)$ the following
way.
\begin{itemize}
\item
Two elements $\underline{a},\underline{b}\in K_k(X)$,
$\underline{a}=\{a_1, a_2,\dots, a_k\}$,
$\underline{b}=\{b_1, b_2,\dots, b_k\}$ are equivalent if there exists
a permutation $\sigma\in S_k$ such that for any subset 
$A=\{i_1< i_2 < \dots <
i_{|A|}\}\in [k]$, $\{a_{i_1}, a_{i_2}, 
\dots, a_{i_{|A|}}\}$  and $\{b_{\sigma(i_1)}, b_{\sigma(i_2)}, 
\dots, b_{\sigma(i_{|A|})}\}$ are both in the same $P^j_{|A|}$ for some
$1\leq j \leq l$.
\end{itemize}
 It is easy to see that this defines an equivalence
relation and thus it results in a partition $\cup^t_{j=1} C_j$
of $K_k(X)$ into {\bf $\cH$-cells}. 
A {\bf cylinder intersection} $L\subset K_r(X)$ is an $r$-uniform hypergraph
defined the following way. Let $B_1$, $B_2$,\dots $B_r$ be 
$r-1$ uniform hypergraphs on $X$, then an $r$-edge $\{a_1, a_2,\dots, a_r\}$
is in $L$ if there exists a permutation $\tau\in S_r$ such that
$$\{a_{\sigma(1)}, a_{\sigma(2)},\dots, a_{\sigma(i-1)}, a_{\sigma(i+1)},
\dots a_{\sigma(r)}\}\in B_i\,.$$
As in the graph case, we call an $r$-uniform hypergraph $G$ $\e$-regular if
$$\Big|\frac{|G|}{|K_r(X)|}-\frac{|G\cap L|}{|L|}\Big| \leq \e\,,$$
for each cylinder intersection $L$, where $|L|\geq \e |K_r(X)|\,.$
Now we are ready to state the hypergraph regularity lemma for $k$-uniform
hypergraphs (see \cite{Gow}, \cite{Ish}, \cite {RSko}, \cite{Tao}).
\begin{theorem}[Hypergraph regularity lemma] Let fix a constant $k>0$. Then
for any $\e>0$ and $F:\bN\to(0,1)$ there exists
constants $c=c(\e,F)$ and $N_0(\e, F)$ such that if $H$ is
a $k$-uniform hypergraph on a set $X$, $|X|\geq N_0(\e,F)$, then there
exists an $F(l)$-equitable $l$-hyperpartition $\cH$ for some
$1<l\leq c$ such that
\begin{itemize}
\item  Each $P^r_j$ is $F(l)$-regular.
\item $|H\triangle T|\leq \e \left({{|X|}\choose{k}}\right)\,$
where $T$ is the union of some $\cH$-cells. 
\end{itemize}
\end{theorem}
\proof
Suppose that the Theorem does not hold for some $e>0$ and $F:\bN\to (0,1)$.
That is there exists a sequence of $k$-uniform hypergraphs $H_i$ without
having $F(j)$-equitable $j$-hyperpartitions for any
$1<j\leq i$ satisfying the conditions of
our Theorem.
Let us consider their ultralimit $[\{S_{H_i}\}^\infty_{i=1}]=\bH\subset \xo^k$.
Similarly to the proof of the Removal Lemma we formulate an infinite
version of the Regularity Lemma as well.

\noindent
Let $K_r(\xo)$ denote the complete $r$-uniform hypergraph on $X$, that
is the set of points \\ $(x_1,x_2,\dots,x_r)\in \xo^r$ such that $x_i\neq x_j$
if $i\neq j$. Clearly $K_r(\xo)\subset \xo^r$ is measurable and
$\mu_{[r]}(K_r(\xo))=1\,.$
An $r$-uniform hypergraph on $\xo$ is an $S_r$-invariant measurable
subset of $K_r(\xo)$. An $l$-hyperpartition $\wch$ is a family of
partitions $K_r(\xo)=\cup^l_{j=1}{\bf  P^j_r}$, where
${\bf P^j_r}$ is an $r$-uniform hypergraph for $1\leq r \leq k$. Again, an
$l$-hyperpartition induces a partition of $K_k(\xo)$ into $\wch$-cells
exactly the same way as in the finite case. It is easy to see that each
$\wch$-cell is measurable.
\begin{proposition} (Hypergraph Regularity Lemma, infinite version)
For any $\e>0$, there exists a $0$-equitable
 $l$-hyperpartition (where $l$ depends on
$\bH$) $\wch$ such that
\begin{itemize}
\item Each ${\bf P^j_r}$ is in $\sigma([r])^*$.
\item $\mu_{[k]}(H\triangle T)\leq \e$, where $T$ is a union of some
$\wch$-cells. \end{itemize}
\end{proposition}
\proof
Let $\cS$ be a separable realization for $\bH$ and
$Q\subseteq [0,1]^{2^k-1}$ be an $S_k$-invariant subset such that
$\mu_{[k]}(F^{-1}(Q)\triangle \bH)=0$. Since $Q$ is a Lebesgue-measurable
set, there exists some $l>0$ such that $Vol_{2^k-1}(Q\triangle Z)<\e$,
where $Z$ is a union of $l$-boxes. Recall that an $l$-box is a product set
in the form
$$\left(\frac{i_1}{l}, \frac{i_1+1}{l}\right)\times 
\left(\frac{i_2}{l}, \frac{i_2+1}{l}\right)
\times \dots \times 
\left(\frac{i_{2^k-1}}{l}, \frac{i_{2^k-1}+1}{l}\right)\,.$$
By the usual symmetrization argument we may suppose
that the set $Z$ is invariant under the $S_k$-action on the $l$-boxes.
Since the measure of points $(x_1,x_2,\dots, x_{2^k_1})\in [0,1]^{2^k-1}$
such that $x_s=x_t$ for some $s\neq t$ is zero, we may also suppose that in 
each box in $Z$, $i_s\neq i_t$ if $s\neq t$.
Let $Z=\cup_{m=1}^q O_m$, where $O_m$ is an $S_k$-orbit of boxes. That is
$O_m=\cup_{\pi\in S_k} \pi(D)$ for some $l$-box $D$. By the previous
condition $\pi_1(D)\neq \pi_2(D)$, if $\pi_1\neq \pi_2$, hence each $O_m$
is the disjoint union of exactly $k!$ $l$-boxes.
 Then
$\mu_{[k]}(F^{-1}(Q)\triangle F^{-1}(Z))<\e$, where $F^{-1}(Z)=
\cup^q_{m=1} F^{-1}(O_m)$\,.
For each $1\leq r \leq k$ we consider the partition 
$\xo^r=\cup_{j=1}^l{\bf  P^j_r}$,
where ${\bf P^j_r}= F^{-1}_{[r]}(\frac{j-1}{l}, \frac{j}{l})\,.$
We call the resulting $l$-hyperpartition $\widetilde{\cH}$.
Note that by the $S_r$-invariance of the separable realization each
${\bf P^j_r}$ is an $r$-uniform hypergraph and also
${\bf P^j_r}\in\sigma([r])^*$.

\begin{lemma} \label{cell}
${\bf C}$ is an $\wch$-cell if and only if
${\bf C}=F^{-1}(\cup_{\pi\in S_k} \pi(D))$, where $D$ is an $l$-box in 
$[0,1]^{2^k-1}$.
\end{lemma}
\proof
By definition $(a_1,a_2,\dots,a_k)\in \xo^k$ and $(b_1,b_2,\dots,b_k)\in
\xo^k$
are in the same $\wch$-cell if and only if
there exists $\pi\in S_k$ such that for any $A\subseteq [k]$
$(a_{i_1},a_{i_2},\dots, a_{i_{|A|}})$ and
$(b_{i_{\pi(1)}},b_{i_{\pi(2)}}\dots,b_{i_{\pi(|A|)}} )$ are in
the same ${\bf  P^j_r}$. That is $F_A(a_1,a_2,\dots,a_k)$ and \\
\noindent
$F_A(b_{\pi(1)}, (b_{\pi(2)},\dots, (b_{\pi(k)}$ are in the
same $l$-box. \qed

\vskip 0.1in
\noindent
Since
$\mu_{[k]}(H\triangle \cup^q_{j=1} F^{-1}(O_m))<\e$, our Proposition
follows. \qed

\vskip 0.2in

\noindent
Now we return to the proof of the Hypergraph Regularity Lemma. First pick an
$r$-hypergraph ${\bf \wp^j_r}$ on $\xo$ such that
$\mu_{[r]}({\bf \wp^j_r}\triangle{\bf  P^j_r})=0$, 
${\bf \wp^j_r}\in \cP_{[r]}$ and
$\cup_{j=1}^l {\bf \wp^j_r}=K_r(\xo)$.
Let $[\{S_{P^j_{r,i}}\}^\infty_{i=1}]={\bf \wp^j_r}\,.$
Then for $\omega$-almost all indices $\cup_{j=1}^l P^j_{r,i}= K_r(X_i)$ is
an $F(l)$-equitable $l$-partition and $|H_i\triangle \cup^q_{m=1} C^i_m|<\e$
for the induced $\cH$-cell approximation.
Here $\cup^q_{m=1} {\bf\widetilde{C}_m}$ is the $\wch$-cell approximation with
respect to the $l$-hyperpartitions $\cup^l_{j=1} {\bf \wp^j_r}=K_r(\xo)$ and
$[\{S_{C^i_m}\}^\infty_{i=1}]={\bf \widetilde{C}_m}$.

\noindent
The only thing remained to be proved is that for $\omega$-almost all
indices $i$ the resulting $l$-hyperpartitions are $F(l)$-regular. If it does
not hold then there exists $1\leq r \leq k$ and $1\leq j \leq l$ such that
for almost all $i$ there exists a cylinder intersection $W_i\subset
K_r(X_i)$,
$|W_i|\geq \e|X_i|$, such that
\begin{equation} \label{egyenlet}
\left|\frac{|P^j_{r,i}|}{|K_r(X_i)|}-\frac{|P^j_{r,i}\cap W_i|}
{|W_i|}\right|>\e\,.
\end{equation}
Let ${\bf W }=[\{S_{W_i}\}^\infty_{i=1}]\,.$ Then
$W\subset \cup_{B\subsetneq [r]} \sigma(B)$. Hence ${\bf \wp^j_r}$ and 
${\bf W}$ are
independent sets. However, by (\ref{egyenlet})
$$\mu_{[r]}({\bf \wp^j_r})\mu_{[r]}({\bf W})\neq \mu_{[r]}({\bf \wp^j_r} 
\cup {\bf W})\,,$$
leading to a contradiction. \qed

\section{Appendix on basic measure theory}
In this section we collect some of the basic results of measure theory
we frequently use in our paper.

\vskip 0.1in
\noindent
\underline{Separable measure spaces:}
Let $(X,\cA,\mu)$ be a probability measure space. Then we call $A,A'\in\cA$ 
equivalent if $\mu(A\triangle A')=0$. The equivalence classes form
a complete metric space, where $d([A],[B])=\mu(A\triangle B)\,.$
This classes form a Boolean-algebra as well, called the
{\bf measure algebra} $\cM(X,\cA,\mu)$. We say that $(X,\cA,\mu)$ is
a {\bf separable} measure space if 
$\cM(X,\cA,\mu)$ is a separable metric 
space. It is important to note that if $(X,\cA,\mu)$ is separable and atomless,
then its measure algebra is isomorphic to the measure algebra of
the standard Lebesgue space $([0,1],\cB,\lambda)$, where $\cB$ is
the $\sigma$-algebra of Borel sets (see e.g. \cite{Hal}.
We use the following folklore version of this
theorem.
\begin{lemma} \label{measurealgebra}
If $(X,\cA,\mu)$ is a separable and atomless measure algebra, then there
exists a map $f:X\to [0,1]$ such that
 $f^{-1}(\cB)\subset \cA$,
$\mu(f^{-1}(U))=\lambda(U)$ for any $U\in \cB$ and
for any $L\in\cA$ there exists
$M\in\cB$ such that $L$ is equivalent to $f^{-1}(M)$. 
\end{lemma}
\proof
Let $I_0$ denote the interval $[0,\frac{1}{2}]$, $I_1=[\frac{1}{2},1]$.
Then let $I_{0,0}=[0,\frac{1}{4}]$, $I_{0,1}=[\frac{1}{4},\frac{1}{2}]$,
$I_{1,0}=[\frac{1}{2},\frac{3}{4}]$, $I_{1,1}=[\frac{3}{4},1]$.
Recursively, we define the dyadic intervals $I_{\alpha_1,\alpha_2,\dots,
\alpha_k}$, where $(\alpha_1,\alpha_2,\dots,
\alpha_k)$ is a $0-1$-string.
Let $T$ be the Boolean-algebra isomorphism between the measure algebra
of $(X,\cA,\mu)$ and the measure algebra of $([0,1],\cB,\lambda)$.
Then we have disjoint sets $U_0, U_1\in\cA$ such that
$T([U_0])=[I_0]$, $T([U_1])=[I_1]$. Clearly $\mu(X\backslash (U_0\cup U_1)=0$.
Similarly, we have disjoint subsets of $U_0$, $U_{0,0}$ and $U_{0,1}$
such that $T([U_{0,0}])=[I_{0,0}]$ and $T([U_{0,1}])=[I_{0,1}]$.
Recursively, we define $U_{\alpha_1,\alpha_2,\dots,
\alpha_k}\in \cA$ such that
$U_{\alpha_1,\alpha_2,\dots,
\alpha_{k-1},0}$ and $U_{\alpha_1,\alpha_2,\dots,
\alpha_{k-1},0}$ are disjoint and $T([U_{\alpha_1,\alpha_2,\dots,
\alpha_k}])= I_{\alpha_1,\alpha_2,\dots,
\alpha_k}$. The set of points in $X$ which are not included in some
$U_{\alpha_1,\alpha_2,\dots,
\alpha_k}$ for some $k>0$ has measure zero.
Now define
$$f(p):=\cap^\infty_{k=1} I_{\alpha_1,\alpha_2,\dots,
\alpha_k}\,,$$
where for each $k\geq 1$, $p\in U_{\alpha_1,\alpha_2,\dots,
\alpha_k}$. It is easy to see that $f$ satisfies the conditions
of our lemma. \qed

\vskip 0.1in
\noindent
\underline{Generated $\sigma$-algebras:} 
Let $(X,\cC,\mu)$ be a probability measure space and $\cA_1,\cA_2,\dots,\cA_k$
be sub-$\sigma$-algebras.
Then we denote by $\langle \cA_i\mid 1\leq i\leq k\rangle$ 
the generated $\sigma$-algebra that
is the smallest sub-$\sigma$-algebra of $\cC$ containing the $\cA_i$'s.
Then the equivalence classes  $$[\cup^n_{j=1} (A^j_1\cap A^j_2\cap\dots
\cap A^j_k)]\,,$$ where $A^j_i\in\cA_i$ and $(A^s_1\cap A^s_2\cap\dots
\cap A^s_k)\cap (A^t_1\cap A^t_2\cap\dots
\cap A^t_k)=\emptyset $ if $s\neq t$
 form a dense subset in the measure algebra
$\cM(X,\langle \cA_i\mid 1\leq i\leq k \rangle,\mu)$ 
with respect to the metric defined above
(see \cite{Hal}).
\vskip 0.1in
\noindent
\underline{Independent subalgebras and product measures:}
The sub-$\sigma$-algebras $\cA_1,\cA_2,\dots,\cA_k\subset \cC$ are 
{\bf independent} subalgebras if
$$\mu(A_1)\mu(A_2)\dots\mu(A_k)=\mu(A_1\cap A_2\cap\dots\cap A_k)\,,$$
if $A_i\in \cA_i$.
\begin{lemma}
\label{fremlin}
Let $\cA_1,\cA_2,\dots,\cA_k\subset \cC$ be independent subalgebras
as above and $f_i:X\to [0,1]$ be maps such that $f_i^{-1}$ defines
isomorphisms between the measure algebras $\cM(X,\cA_i,\mu)$ and
$\cM([0,1],\cB,\lambda)$. Then the map $F^{-1}$, $F=\oplus_{i=1}^k f_i:X\to
[0,1]^k$ defines an isomorphism between the measure algebras
$\cM(X,\langle \cA_i\mid 1\leq i\leq k \rangle,\mu)$ and
$\cM([0,1]^k,\cB^k,\lambda^k)$.
\end{lemma}
\proof
Observed that
$$\mu (F^{-1}(\cup_{i=1}^s [A^i_1\times\dots\times A^i_k]))=\sum^s_{i=1}
\lambda^k[A^i_1\times\dots\times A^i_k]$$
whenever $\{A^i_1\times\dots\times A^i_k\}^s_{i=1}$ are disjoint product
sets. Hence $F^{-1}$ defines an isometry between dense subsets of the two
measure algebras. \qed

\vskip 0.1in
\noindent
\underline{Radon-Nykodym Theorem:} 
Let $(X,\cA,\mu)$ be a probability measure space and $\nu$ be 
an absolutely continuous measure with respect to $\mu$. That is
if $\mu(A)=0$ then  $\nu(A)=0$ as well.
Then there exists an integrable $\cA$-measurable function $f$
such that
$$\mu(A)=\int_A f d\mu$$
for any $A\in\cA$.

\vskip 0.1in
\noindent
\underline{Conditional expectation:} 
Let $(X,\cA,\mu)$ be a probability measure space
and $\cB\subset\cA$ be a sub-$\sigma$-algebra. Then by the
Radon-Nykodym-theorem for any integrable
$\cA$-measurable function $f$ there exists an integrable
$\cB$-measurable function $E(f\mid\cB)$ such that
$$\int_B E(f\mid\cB) d\mu=\int_B f d\mu\,,$$
if $B\in\cB$. The function $E(f\mid \cB)$ is called the conditional
expectation of $f$ with respect to $\cB$. It is unique up to a
zero-measure perturbation.
Note that if $a\leq f(x)\leq b$ for almost all $x\in X$, then
$a\leq E(f\mid \cB)(x)\leq b$ for almost all $x\in X$ as well.
Also, if $g$ is a bounded $\cB$-measurable function, then
$$E(fg\mid\cB)=E(f\mid \cB) g\,\,\quad\mbox{almost everywhere}\,.$$
The map $f\to E(f,\cB)$ extends to a Hilbert-space projection
$E:L^2(X,\cA,\mu)\to L^2(X,\cB,\mu)$.

\vskip 0.1in
\noindent
\underline{Lebesgue density theorem:} 
Let $A\in \bR^n$ be a measurable set. Then almost all points $x\in A$ is
a {\bf density point}. The point $x$ is a density point if
$$\lim_{r\to 0} \frac{Vol (B_r(x)\cap A)} {Vol (B_r(x))}=1\,,$$
where $Vol$ denotes the $n$-dimensional Lebesgue-measure.

\noindent
G\'abor Elek
\noindent
Alfred Renyi Institute of the Hungarian Academy of Sciences
\noindent
POB 127, H-1364, Budapest, Hungary,  elek@renyi.hu

\vskip 0.2in

\noindent
Bal\'azs Szegedy
\noindent
University of Toronto, Department of Mathematics,
\noindent
St George St. 40, Toronto, ON, M5R 2E4, Canada

\end{document}